\documentclass[11pt]{article}
\usepackage[margin=1in]{geometry}

\usepackage{amsmath,amsthm,amsfonts,amssymb}
\usepackage{graphicx}
\usepackage{url}
\usepackage{lineno}
\usepackage{hyperref}
\usepackage{rotating}

\usepackage{xcolor}

\title{Some $i$\textsc{-Mark} games}
\date{}
\author{Oren Friman\thanks{\texttt{orenfriman@gmail.com}. Ariel University, Ariel, Israel.} \and Gabriel Nivasch\thanks{ \texttt{gabrieln@ariel.ac.il}. Ariel University, Ariel, Israel.}}

\DeclareMathOperator{\SG}{SG}
\DeclareMathOperator{\mex}{mex}
\DeclareMathOperator{\opt}{opt}
\newcommand{\imark}{i\textsc{-Mark}}
\newcommand{\boldimark}{\texorpdfstring{$i$-Mark}{i-Mark}}
\newcommand{\Mark}{\textsc{Mark}}
\newcommand{\markt}{\textsc{Mark-}t}
\newcommand{\upmark}{\textsc{UpMark}}
\newcommand{\N}{{\mathbb N}}

\newtheorem{theorem}{Theorem}
\newtheorem{lemma}[theorem]{Lemma}

\newtheorem{conjecture}[theorem]{Conjecture}

\begin{document}
\maketitle

\begin{abstract}
Let $S$ be a set of positive integers, and let $D$ be a set of integers larger than $1$. The game $\imark(S,D)$ is an impartial combinatorial game introduced by Sopena (2016), which is played with a single pile of tokens. In each turn, a player can subtract $s \in S$ from the pile, or divide the size of the pile by $d \in D$, if the pile size is divisible by $d$. Sopena partially analyzed the games with $S=[1, t-1]$ and $D=\{d\}$ for $d \not\equiv 1 \pmod t$, but left the case $d \equiv 1 \pmod t$ open. 

\begin{sloppypar}
We solve this problem by calculating the Sprague--Grundy function of $\imark([1,t-1],\{d\})$ for $d \equiv 1 \pmod t$, for all $t,d \geq 2$. We also calculate the Sprague--Grundy function of $\imark(\{2\},\allowbreak\{2k + 1\})$ for all $k$, and show that it exhibits similar behavior. Finally, following Sopena's suggestion to look at games with $|D|>1$, we derive some partial results for the game $\imark(\{1\}, \{2, 3\})$, whose Sprague--Grundy function seems to behave erratically and does not show any clear pattern. We prove that each value $0,1,2$ occurs infinitely often in its SG sequence, with a maximum gap length between consecutive appearances.
\end{sloppypar}

\end{abstract}

\section{Introduction}

The game of $\imark$ is an impartial combinatorial game introduced by Sopena~\cite{SOPENA201690}. Given a nonempty set $S$ of positive integers and a nonempty set $D$ of integers larger than $1$, the game $\imark(S,D)$ is played with a single pile of tokens. In each turn, a player can subtract $s \in S$ from the current pile, or divide the size of the pile by $d \in D$, if the pile size is divisible by $d$. Following the normal convention on impartial games, we will assume that the last player to move is the winner.

\subsection{Other subtraction-division games}

The game of $\imark$ is a \emph{subtraction-division game} (where the \emph{i} stands for \emph{integral}). The first subtraction-division game studied was $\Mark$ (Berlekamp and Buhler~\cite{PUUZZLE2009}), in which the pile size~$n$ can be replaced by $n-1$ or $\lfloor n/2\rfloor$. This game was analyzed by Fraenkel~\cite{APERIODIC,VILEANDDOPEY} and Guo~\cite{guo2012winning}. Fraenkel also studied a variant $\upmark$ in which one rounds up instead of down, as well as the game $\markt$, in which one can subtract any number between $1$ and $t-1$ or divide by $t$ rounding down.

One can define a general class of games $SD(S,D)$, in which the pile size $n$ can be replaced by $n-s$ for $s\in S$ or $\lfloor n/d \rfloor$ for $d\in D$. Hence, $\Mark=SD(\{1\},\{2\})$, and $\markt=SD([1,t-1],\{t\})$. 

Kupin~\cite{kupin2011subtraction} considered a variant of $SD(S,D)$ in which the division move is to $\lceil n/d\rceil$ and the game ends at $n=1$.

\subsection{Background on impartial games}

An impartial game can be abstractly represented by a directed acyclic graph $G=(V,E)$, in which the vertices $V$ represent positions and the directed edges $E$ represent legal moves. If $(v,w)\in E$ we say that $w$ is an \emph{option} of $v$. We denote the set of options of $v$ by $\opt(v)$.

The positions in an impartial game are classified into \emph{$P$-positions} and \emph{$N$-positions}. From a $P$-position, if both players play optimally then the second (or \emph{previous}) player will win, whereas from an $N$-position, the first (or \emph{next}) player will win. The $P$- or $N$-status of a position is known as its \emph{outcome}. Under the normal play convention (in which the last player to move is the winner), the outcome is recursively characterized as follows: A position is a $P$-position if and only if all its options are $N$-positions. In particular, terminal positions are $P$-positions.

Given games $G_1, \ldots, G_n$, their \emph{sum} $G = G_1 + \cdots + G_n$ is defined as the game with the following rules: In each turn, a player chooses some $G_i$ and makes some legal move on it, leaving all other games untouched. The game $G$ ends when all the component games $G_i$ end.

Knowing the outcomes of the component games $G_i$ is not enough to determine the outcome of their sum. In fact, if $G=G_1+G_2$ and $v_1\in G_1$, $v_2\in G_2$ are both $N$-positions, then $(v_1,v_2)\in G$ could be either a $P$- or an $N$-position.

In order to be able to optimally play sums of games (under normal play convention), one needs a generalization of the notion of $P$- and $N$-positions known as the \emph{Sprague--Grundy function} (or \emph{SG function} for short) \cite{GRUNDY,SPRAGUE}. The SG value of a position in a game is a nonnegative integer. Given a finite set $S\subset \N$, define $\mex(S) = \min{(\N\setminus S)}$ (mex stands for \emph{minimum excluded value}). Then, for each position $v\in G$  recursively define
\begin{equation}\label{eq_S_def}
\SG(v) = \mex\{\SG(w):w\in\opt(v)\}.
\end{equation}
The SG function has the following two important properties:
\begin{itemize}
\item $\SG(v)=0$ if and only if $v$ is a $P$-position.
\item Let $a\oplus b$ denote the bitwise XOR of $a$ and $b$ written in binary. Let $v = (v_1, \ldots, v_n)\in G_1+\cdots+G_n$. Then $\SG(v) = \SG(v_1) \oplus \cdots \oplus \SG(v_n)$.
\end{itemize}

Hence, knowledge of the SG function of the component games allows us to optimally play their sum.

\subsection{Polynomial-time algorithms}

Let $G$ be a game played on a single pile of tokens. A position $n$ in the game can be represented with $O(\log n)$ bits. Hence, computing $\SG(n)$ using the definition (\ref{eq_S_def}) takes exponential time in the size of the input. In order to be able to play the game efficiently, we need a polynomial-time algorithm for computing $\SG(n)$. One way to do this is to find an explicit characterization of $\SG(n)$, which immediately yields a polynomial-time algorithm.

\subsection{\boldmath Sopena's results on \boldimark}

In~\cite{SOPENA201690}, Sopena presented several results on $\imark$. Regarding the normal-play version, he proved that, although the SG sequence of $\imark$ is aperiodic for every $S$ and $D$, the outcome sequence for $\imark([1, t - 1],\{d\})$ with $t, d \geq 2$ and $d \not\equiv 1 \pmod t$  is periodic, with set of $P$-positions equal to
\begin{equation*}
    \mathcal P = \{ qt \mid 0 \leq q < d \} \cup \{ qt + 1 \mid q \geq d \}.
\end{equation*}
Sopena also proved that the Sprague--Grundy sequence is \emph{1-almost periodic} if $d = 2$ or $d = t$, and that for these cases the characterization can be calculated by a linear-time algorithm.
He also studied the games $\imark(\{a, 2a\},\{2\})$ for $a \geq 1$. Sopena raised several open questions, which include among them:
\begin{itemize}
\item Do there exist $S$ and $D$ for which the outcome sequence of $\imark(S,D)$ is not periodic?
\item What can be said about $\imark([1, t - 1],\{d\})$ for $d\equiv 1\pmod t$?
\item What can be said about games $\imark(S,D)$ with $|D|>1$?
\end{itemize}

\subsection{Our contributions}
In this paper we address the above-mentioned questions raised by Sopena. First, we give a complete characterization of the SG sequence of $\imark([1, t - 1],\{d\})$ for $d\equiv 1\pmod t$ and $t,d\ge 2$. As we show, the pile sizes whose SG value is $t$ form two geometric sequences with ratio $d$, while the values $0, \ldots, t - 1$ repeat in a periodic pattern between them. See Figure~\ref{fig_patterns}(i) and Section~\ref{sec_t}.

Our second contribution concerns games of the form $\imark(\{2\},\{k\})$ for $k$ odd. Since in these games each position has at most two options, their SG function is bounded by $2$. We find that the positions with SG value $2$ form one or two geometric progressions, similarly to the case above. Some of these phenomena were previously observed by Choen and Gerchikov, as part of an undergraduate project \cite{cg_report}. See Figure~\ref{fig_patterns}(ii--iii) and Section~\ref{sec_2}.

In all the above games, the outcome sequence is aperiodic.

In Section~\ref{sec_2} we also state a conjecture which partially characterizes the SG sequence of all games of the form $\imark(\{s\},\{d\})$ for $s,d$ relatively prime.

Our third and final contribution concerns the game $\imark(\{1\},\{2, 3\})$. This is arguably the simplest instance of a game in which $\lvert D \rvert > 1$. The SG sequence of this game seems to behave erratically, and we were unable to identify a precise pattern. We still try to establish some basic properties of this sequence, in the same spirit as previous works on other games with seemingly chaotic SG functions, such as the game of Wythoff \cite{nivasch2005more}, and the 3-row case of the game of Chomp \cite{friedman2005geometry}.
We observe experimentally that for each $0 \leq i \leq 3$, instances of the value $i$ in the sequence are never missing for too long, meaning, for every $i$ there seems to exist a constant $c_i$ such that every interval $[n + 1, n + c_i]$ contains a position with SG value $i$. We prove the existence of $c_0, c_1$ and $c_2$, we leave the existence of $c_3$, as well as a more precise characterization of the sequence, as open problems. See Section~\ref{sec_123}.

\section{\boldmath Analysis of \boldimark\texorpdfstring{${([1, t - 1],\{d\})}$}{([1,t-1],\{d\})} for \texorpdfstring{${d\equiv 1\pmod t}$}{d=1 (mod t)}}\label{sec_t}

\begin{sidewaysfigure}
\centerline{\includegraphics{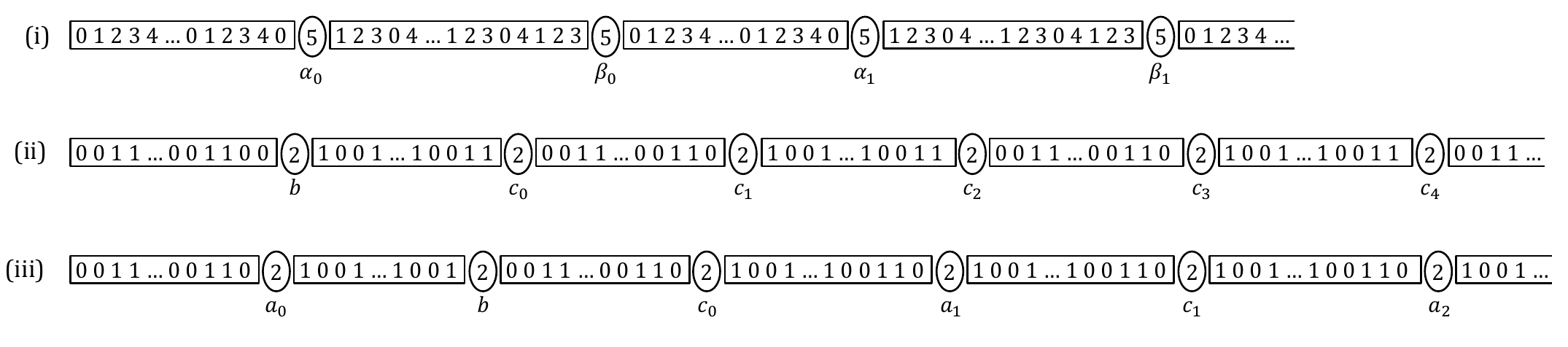}}
\caption{\label{fig_patterns}(i) SG sequence of the game $\imark([1,t-1],\{d\})$, for $d\equiv 1 \pmod t$. Here $t=5$. (ii) SG sequence of the game $\imark(\{2\},\{k\})$, for $k\equiv 3 \pmod 4$. (iii) SG sequence of the game $\imark(\{2\},\{k\})$, for $k\equiv 1 \pmod 4$.}
\end{sidewaysfigure}

\begin{theorem}
Let $\SG(n)$ denote the Sprague--Grundy value of a pile of $n$ tokens in the game $\imark([1,t-1],\{d\})$, where $t,d\ge 2$ and $d\equiv 1 \pmod t$. For $m\ge 0$ let
\begin{align*}
\alpha_m &= d + \cdots + d^{m+1} = (d^{m+2}-d)/(d-1),\\
\beta_m &= d + \cdots + d^m + td^{m+1} = td^{m+1} + (d^{m+1}-d)/(d-1).
\end{align*}
Note that $\alpha_0 < \beta_0 < \alpha_1 < \beta_1 < \cdots$, and that $\alpha_m \equiv m+1 \pmod t$ and $\beta_m \equiv m \pmod t$. Let
\begin{align*}
A_0 &= [0,\alpha_0-1];\\
A_m &= [\beta_{m-1} + 1, \alpha_m-1],\qquad m\ge 1;\\
B_m &= [\alpha_m+1,\beta_m-1],\qquad m\ge 0.
\end{align*}
Then:
\begin{itemize}
\item $\SG(\alpha_m) = \SG(\beta_m) = t$ for all $m\in\N$.
\item The SG sequence of $A_m$ is
\begin{equation*}
(0,1,\ldots, t-1)^{z_m}, 0,
\end{equation*}
for $z_m = (\alpha_m - \beta_{m-1} - 2)/t$.
\item The SG sequence of $B_m$ is
\begin{equation*}
(1,2, \ldots, t-2,0,t-1)^{z'_m},1,\ldots,t-2,
\end{equation*}
for $z'_m  = (\beta_m  - \alpha_m - t + 1)/t$.
\end{itemize}
Here $X^z$ denotes $z$-fold repetition of the sequence $X$.
\end{theorem}

\begin{proof}
By induction on $n$. Since $\alpha_0=d$, if $n\in A_0$ then $n$ has only subtraction options, and the claim follows easily.

If $n=\alpha_0$ then $n/d=1$, which has SG value $1$, so $\SG(\alpha_0) = t$. If $n=\beta_0=td$ then the subtraction options of $n$ have SG values $t-1,1,\ldots,t-2$, while its division option is $t\in A_0$, which has SG value $0$. Hence, $\SG(\beta_0)=t$.

Suppose $n$ equals $\alpha_m$ or $\beta_m$ for $m\ge 1$. The numbers $\alpha_m, \beta_m$ are divisible by $d$, and they satisfy
\begin{equation}\label{eq_alpha_beta_divide}
\alpha_{m+1}/d = \alpha_m+1,\qquad \beta_{m+1}/d = \beta_m+1.
\end{equation}
Therefore, the subtraction options of $\alpha_m$ have SG values $2,\ldots, t-1, 0$, while its division option has SG value $1$, so $\SG(\alpha_m) = t$. Similarly, the subtraction options of $\beta_m$ have SG values $t-1,1,\ldots,t-2$, while its division option has SG value $0$, so $\SG(\beta_m) = t$.

Next, suppose $n=\alpha_m+i$ for $1\le i \le t-2$. Then the subtraction options of $n$ are the last $t-1-i$ elements of $A_m$; $\alpha_m$; and the first $i-1$ elements of $B_m$. Their respective SG values are $i+2, \ldots, t-1,0$; $t$; and $1,\ldots, i-1$. Hence, $\SG(n) = i$.

The subtraction options of $n=\alpha_m+t-1$ are $\alpha_m$ and the first $t-2$ elements of $B_m$. Their respective SG values are $t,1,\ldots,t-2$. Hence, $\SG(\alpha_m+t-1) = 0$.

Similarly, for $1\le i\le t-1$, the subtraction options of $n=\beta_m+i$ are the last $t-1-i$ elements of $B_m$; $\beta_m$; and the first $i-1$ elements of $A_{m+1}$. Their respective SG values $i,\ldots,t-2$; $t$; and $0, \ldots, i-2$. Hence, $\SG(n) = i-1$.

Next, suppose $n$ is the $i$th element of $A_m$ or $B_m$ for $i\ge t$. Suppose first that $n$ is not divisible by $d$. Then the subtraction options of $n$ are the $t-1$ preceding elements of $A_m$ or $B_m$. Their SG values equal $\{0,\ldots,t-1\}\setminus\{j\}$ for some $j$, and so $\SG(n)=j$.

Now suppose $n$ is divisible by $d$. Since the SG sequences of $A_m, B_m$ are periodic with period $t$, and since $d\equiv 1 \pmod t$, advancing $d$ steps in $A_m$ or $B_m$ is equivalent to advancing one step. By (\ref{eq_alpha_beta_divide}) and induction on $i$, it follows that if the $i$th element of $A_m$ (resp.~$B_m$) is divisible by $d$, then its division option is in a position congruent to $i+1$ modulo $t$ in $A_{m-1}$ (resp.~$B_{m-1}$). Therefore, if $n\neq \alpha_m-d$ (resp.~$n\neq \beta_m-d$), then the SG value of the division option of $n$ is included in the SG values of the subtraction options of $n$, so it has no effect on $\SG(n)$. Finally, if $n$ equals $\alpha_m-d$ (resp.~$\beta_m-d$) then $n/d$ equals $\alpha_{m-1}$ (resp.~$\beta_{m-1}$), whose SG value is $t$, which again has no effect on $\SG(n)$.
\end{proof}

\section{\boldmath Analysis of \boldimark\texorpdfstring{$(\{2\},\{k\})$}{(\{2\},\{k\})} for \texorpdfstring{$k$}{k} odd}\label{sec_2}

Here the behavior depends on whether $k\bmod 4$ equals $1$ or $3$.

\begin{theorem}
Let $k\equiv 3 \pmod 4$, and let $\SG(n)$ denote the Sprague--Grundy value of a pile of $n$ tokens in the game $\imark(\{2\},\{k\})$. Let $b = 2k$, $c_0 = 4k$, and $c_m = k(c_{m-1}+2)$ for $m\ge 1$. Let
\begin{align*}
I_0 &= [0,b-1];\\
I_1 &= [b+1,c_0-1];\\
I_m &= [c_{m-2} + 1, c_{m-1}-1],\qquad m\ge 2.
\end{align*}
Then:
\begin{itemize}
\item $\SG(b)=2$ and $\SG(c_m) = 2$ for all $m\in\N$.
\item The SG sequence of $I_0$ is $(0,0,1,1)^{z_0},0,0$ for $z_0=(b-2)/4$.
\item For $m$ odd, the SG sequence of $I_m$ is $(1,0,0,1)^{z_m},1$, where $z_1 = (c_0-b-2)/4$ and $z_m = (c_{m-1} - c_{m-2} - 2)/4$ for $m\ge 3$.
\item For $m\ge 2$ even, the SG sequence of $I_m$ is $(0,0,1,1)^{z_m},0$ for $z_m = (c_{m-1} - c_{m-2} - 2)/4$.
\end{itemize}
\end{theorem}

\begin{proof}
The claim up to $n=c_0$ is easily checked.

Consider $n=c_m$, $m\ge 1$. Its subtraction option has SG value $1$, and its division option is the second term of $I_{m+1}$, which has SG value $0$, so $\SG(c_m) = 2$.

Now suppose $n\in I_m$, $m\ge 2$. If $n$ is not divisible by $k$ then only the subtraction option of $n$ is present, and the claim is easily checked (including the cases in which $n-2$ equals $c_{m-2}$ or belongs to $I_{m-1}$).

Hence, suppose $n$ is divisible by $k$, and let $z=n/k$ be the division option of $n$. If $m=2$ then we can have $z\in I_0$ (if $n<2k^2$), or $z=b$ (if $n=2k^2)$, or $z\in I_1$ (if $2k^2<n<c_1-2k$), or $z = c_0$ (if $n=c_1-2k$), or $z=c_0+1\in I_2$ (if $n=c_1-k$). If $m\ge 3$ then only the last three options are possible, namely $z\in I_{m-1}$ (if $n<c_{m-1}-2k$), or $z=c_{m-2}$ (if $n = c_{m-1} - 2k$), or $z = c_{m-2}+1\in I_m$ (if $n=c_{m-1}-k$).

Since the SG sequences of $I_m$ and $I_{m-1}$ are periodic with period $4$, and $k\equiv 3 \pmod 4$, increasing $n$ by $k$ in $I_m$ is equivalent to increasing it by $3$, and then $z$ increases by $1$.

Thus, when $m=2$ and $n$ starts at $c_0+k = 5k$, the subtraction option of $n$ goes through the SG values $0,1,1,0,\ldots$, and its division option $z\in I_0$ also goes through SG values $0,1,1,0,\ldots$, which implies that $\SG(n)$ goes through the values $1,0,0,1,\ldots$, as desired. Once $n=2k^2\equiv 2 \pmod 4$, its subtraction option has SG value $1$, while its division option $z=b$ has SG value $2$, so $\SG(n) = 0$, as desired. As $n$ continues from $2k^2+k$, the subtraction option of $n$ goes through the SG values $1,0,0,1,\ldots$, and its division option $z\in I_1$ also goes through SG values $1,0,0,1,\ldots$, which implies that $\SG(n)$ goes through the values $0,1,1,0,\ldots$, as desired. Once $n$ reaches $c_1-2k$, its subtraction option has SG value $0$, and its division option $z=c_0$ has SG value $2$, so $\SG(n) = 1$. Finally, when $n = c_1-k$, its subtraction option has SG value $0$, while $\SG(z) = \SG(c_0+1) = 0$, so $\SG(n) = 1$.

When $m\ge 3$ is odd, the subtraction option of $n$ goes through SG values $1,1,0,0,\ldots$, and its division option also goes through SG values $1,1,0,0,\ldots$, which implies that $\SG(n)$ goes through the values $0,0,1,1,\ldots$, as desired. Once $n$ reaches $n = c_{m-1}-2k\in I_m$, the subtraction option of~$n$ has SG value $0$, while its division option is $c_{m-2}$ which has SG value $2$, so $\SG(n) = 1$. When $n = c_{m-1} - k \in I_m$, the subtraction option of $n$ has SG value $1$, while its division option is the first element of $I_m$, which also has SG value $1$, so $\SG(n)=0$ as desired.

The case $m\ge 4$ even is like the last three subcases of $m=2$.
\end{proof}

\begin{theorem}
Let $k\equiv 1 \pmod 4$, $k>1$, and let $\SG(n)$ denote the Sprague--Grundy value of a pile of $n$ tokens in the game $\imark(\{2\},\{k\})$. Let $b = 2k$, $c_0 = 4k$, and $c_m = k(c_{m-1}+2)$ for $m\ge 1$. Let $a_0 = k$ and $a_m = k(a_{m-1}+2)$ for $m\ge 1$. Note that
\begin{equation*}
a_0 < b < c_0 < a_1 < c_1 < a_2 <\cdots,
\end{equation*}
since $a_i - c_j = k(a_{i-1} - c_{j-1})$. Let
\begin{align*}
X &= [0,a_0-1];\\
Y &= [a_0+1,b-1];\\
Z &= [b+1,c_0-1];\\
A_m &= [c_{m-1} + 1, a_m-1],\qquad m\ge 1;\\
C_m &= [a_m+1,c_m-1],\qquad m\ge 1.
\end{align*}
Then:
\begin{itemize}
\item $\SG(b) = 2$, and $\SG(a_m) = \SG(c_m) = 2$ for all $m\in\N$;
\item The SG sequence of $X$ is $(0,0,1,1)^z,0$ for $z=(a_0-1)/4$.
\item The SG sequence of $Y$ is $(1,0,0,1)^{z'}$ for $z' = (b-a_0-1)/4$.
\item The SG sequence of $Z$ is $(0,0,1,1)^{z''},0$ for $z'' = (c_0 - b - 2)/4$.
\item For every $m\ge 1$, the SG sequence of $A_m$ is $(1,0,0,1)^{y_m},1,0$ for $y_m = (a_m - c_{m-1} - 3)/4$.
\item For every $m\ge 1$, the SG sequence of $C_m$ is $(1,0,0,1)^{y'_m},1,0$ for $y'_m = (c_m - a_m - 3)/4$.
\end{itemize}
\end{theorem}

\begin{proof}
It is tedious but straightforward to verify the claim up to $n=c_1-1$.

If $n$ equals $a_m$ (resp.~$c_m$), then the subtraction option of $n$ has SG value $1$, while its division option is the second element of $C_{m-1}$ (resp.~$A_m$), which has SG value $0$. Hence, $\SG(n) = 2$.

Suppose $n$ belongs to $A_m$ or $C_m$ for $m\ge 2$. If $n$ is not divisible by $k$, then only the subtraction option of $n$ is present and the claim is readily verified. Hence, suppose $n$ is divisible by $k$. The SG sequences of $A_m$ and $C_m$ are periodic with period $4$, and $k\equiv 1 \pmod 4$, so advancing $n$ by $k$ steps is equivalent to advancing it by one step. Hence, starting from $n=c_{m-1}+k\in A_m$ or $n=a_m+k\in C_m$, and advancing $n$ in steps of $k$, the subtraction option of $n$ goes through the values $0,1,1,0,\ldots$, and its division option goes through exactly the same values, implying that $\SG(n)$ goes through the values $1,0,0,1,\ldots$, as desired.

If $n=a_m-2k\in A_m$ (resp.~$n=c_m-2k\in C_m$), then the subtraction option of $n$ has SG value $0$, while its division option is $a_{m-1}$ (resp.~$c_{m-1}$), which has SG value $2$. Hence, $\SG(n) = 1$, as desired.

Finally, if $n=a_m-k\in A_m$ (resp.~$n=c_m-k\in C_m$), then the subtraction option of $n$ has SG value $1$, while its division option is the first element of $C_{m-1}$ (resp.~$A_m$), which also has SG value~$1$. Hence, $\SG(n) = 0$, as desired.
\end{proof}

We conclude this section with a conjecture supported by our experiments:

\begin{conjecture}
Let $n$ be any position in the game i-\textsc{Mark($\{s\},\{d\}$)} in normal play, where $s$ and $d$ are relatively prime. Then a necessary condition to have $\SG(n) = 2$ is for one of the following properties to hold:
\begin{align*}
   &n = s d, \\
   &n\in (a_n) \text{ where } a_0 = 2 d s \text{ and } a_{i + 1} = d (a_i + s), \\
   &n\in (b^1_n) \text{ where } b^1_0 = d \text{ and } b^1_{i + 1} = d (b^1_i + s), \\
   &\;\;\vdots \notag\\
   &n\in (b^{s - 1}_n) \text{ where } b^{s - 1}_0 = (s - 1)d \text{ and } b^{s - 1}_{i + 1} = d (b^{s - 1}_i + s).
\end{align*}
The rest of the SG values iterate in $s$-tuples of zeros and ones.
\end{conjecture}

\section{\boldmath Some results on \boldimark\texorpdfstring{$(\{1\},\{2,3\})$}{(\{1\},\{2,3\})}}\label{sec_123}

In this section we take the first steps towards studying $\imark$ games with more than one division option.
Specifically, we look at the game $\imark(\{1\},\{2,3\})$, which is arguably the simplest such case. The SG sequence of this game is composed of the values $0,1,2,3$; however, the sequence seems chaotic and does not show any clear pattern.

We used a computer program to calculate the maximum gap between consecutive appearances of each SG value. See Table~\ref{tab:max_gap}.

\begin{table}
\centering
\begin{tabular}{r|c c c c} 
 SG value    & 0 & 1 & 2  & 3 \\\hline 
 maximum gap & 4 & 8 & 19 & 240
\end{tabular}
\caption{Maximum gap for each $\SG$ value for $n \leq 2^{31}-1$ in $\imark(\{1\},\{2,3\})$.}
\label{tab:max_gap}
\end{table}

These computer experiments seem to indicate that for each $0\le i\le 3$ there exists a constant $c_i$ such that every interval $[n+1,n+c_i]$ contains at least one position with SG value $i$. In this section we prove the existence of the constants $c_0, c_1, c_2$, and leave the existence of $c_3$ as an open problem. For $c_0$ we obtain the optimal value $c_0=4$.

The options of a position $n$ in this game are determined by the remainder of $n$ modulo $6$. Recall that $\SG(n) \leq \lvert \opt(n) \rvert$ for every $n$, and see Table~\ref{tab:mod_options}.

\begin{table}
\centering
\begin{tabular}{c|c|c} 
 $n\bmod 6$         & options & max $\SG(n)$ \\ 
  \hline
 0 & $n - 1, \frac{n}{2}, \frac{n}{3}$  & $3$ \\
 1 & $n - 1$ & $1$  \\
 2 & $n - 1, \frac{n}{2}$ & $2$  \\
 3 & $n - 1, \frac{n}{3}$& $2$   \\
 4 & $n - 1, \frac{n}{2}$& $2$   \\
 5 & $n - 1$ & $1$
\end{tabular}
\caption{Options for each position in $\imark(\{1\},\{2,3\})$.}
\label{tab:mod_options}
\end{table}

\begin{theorem}\label{gap_0}
Let $n > 0$ be a position in $\imark(\{1\},\{2,3\})$. Then there exists $i\in\N$, $1 \leq i \leq 4$ such that $\SG(n - i) = 0$.
\end{theorem}
\begin{proof}
Let $n > 0$. We clearly have $\SG(0) = 0$, which implies that for $n \leq 4$ the Theorem holds. Now let $n > 4$ and let $i>0$ be smallest such that $n - i \equiv 1 \pmod 6$ or $n - i \equiv 5 \pmod 6$. Hence,
\begin{equation*}
 i = \begin{cases}
    1 ; & n \equiv 0 \pmod 6, \\
    2 ; & n \equiv 1 \pmod 6, \\
    1 ; & n \equiv 2 \pmod 6, \\
    2 ; & n \equiv 3 \pmod 6, \\
    3 ; & n \equiv 4 \pmod 6, \\
    4 ; & n \equiv 5 \pmod 6.
  \end{cases}
\end{equation*}
By Table \ref{tab:mod_options}, $\opt(n - i) = \{ n - i - 1 \}$, so $\SG(n - i) \leq 1$. This implies that either $\SG(n - i) = 0$ or $\SG(n - i - 1) = 0$. Therefore the maximum gap is at most $i + 1$.
Hence, we have proven the claim for all cases except for $n \equiv 5 \pmod 6$.

Now let $n \equiv 5 \pmod 6$, and suppose for a contradiction that $\SG(n-i)\neq 0$ for all $1\le i\le 4$. Since $n-4\equiv 1 \pmod 6$, by Table \ref{tab:mod_options} we must have  so $\SG(n - 4) = 1$. Hence, we must have $\SG(n-3) = 2$, and so $\SG(n-2) = 1$, and $\SG(n-1)=2$. But then $\SG((n-3)/2) = \SG((n-1)/2) = 0$, which is a contradiction.
\end{proof}

\begin{theorem}
Let $n > 1$ be a position in $\imark(\{1\},\{2,3\})$. Then there exists $1 \leq i \leq 10$ such that $\SG(n - i) = 1$.
\end{theorem}

\begin{figure}
\centerline{\includegraphics{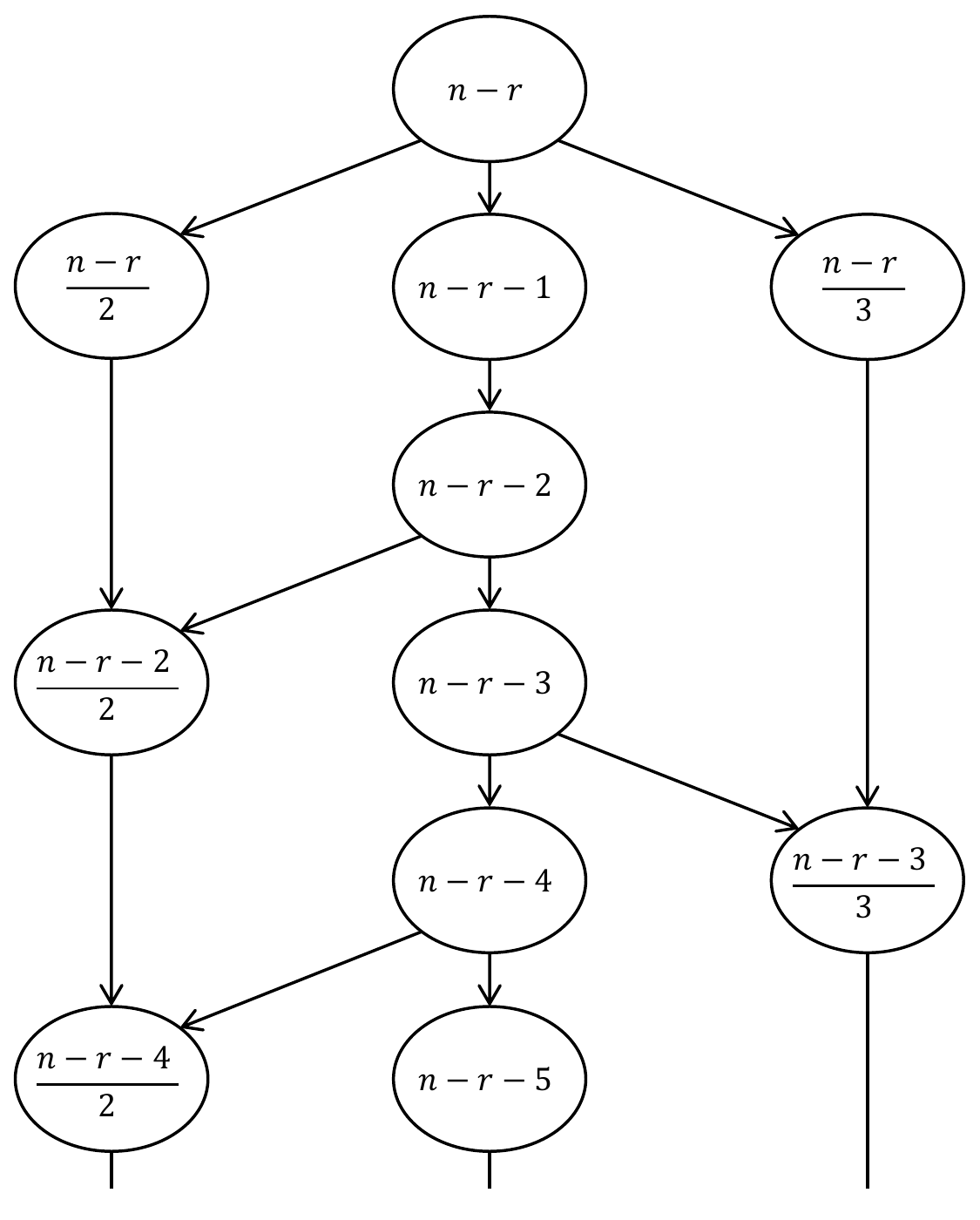}}
\caption{\label{fig_123}Positions reachable from position $n - r$ where $n \equiv r \pmod 6$}
\end{figure}

\begin{proof}
Since $\SG(1) = 1$, the claim holds for $n \leq 11$, so let $n > 11$. Let $r = n \bmod 6$. We will show that $\SG(n-r-i) = 1$ for some $1\le i\le 5$, so suppose for a contradiction that this is not the case. See Figure~\ref{fig_123}.

We have $(n - r - 1) \equiv 5 \pmod 6$ so by Table \ref{tab:mod_options} we must have $\SG(n - r - 1) = 0$. Then $\SG(n - r - 2) = 2$, $\SG(n - r - 3) = 0$, $\SG(n - r - 4) = 2$, and $\SG(n-r-5)=0$. But then $\SG((n - r - 2) / 2) = \SG((n - r - 4) / 2) = 1$, which is a contradiction.
\end{proof}

\begin{theorem}
Let $n > 3$ be a position in $\imark(\{1\},\{2,3\})$. Then there exists $1 \leq i \leq 49$ such that $\SG(n - i) = 2$.
\end{theorem}
\begin{proof}
Since $SG(3) = 2$, the claim holds for $n \leq 52$, so let $n > 52$. We first prove the following Lemma:
\begin{lemma}\label{temp_gap}
Suppose $m \equiv 5 \pmod 6$ and suppose $\SG(m - i) \neq 2$ for all $0 \leq i \leq 7$. 
Then $\SG(m) = 0$. 
\end{lemma}
\begin{proof}
Suppose by contradiction that $\SG(m) \neq 0$. By Table \ref{tab:mod_options} we must have $\SG(m) = 1$, $\SG(m - 1) = 0$, $\SG(m - 2) = 1$, $\SG(m - 3) = 0$, $\SG(m - 4) = 1$, $\SG(m - 5) = 0$, $\SG(m - 6) = 1$, $\SG(m - 7) = 0$. Therefore $\SG((m - 1 - 2j) / 2) \neq 0$ for all $0 \leq j \leq 3$. This contradicts Theorem \ref{gap_0}.
\end{proof}

Now given $n > 52$, suppose for a contradiction that $\SG(n - i) \neq 2$ for all $1 \leq i \leq 49$. Let $r = n \bmod 36$. Hence $r \leq 35$ and $n - r \equiv 0 \pmod{36}$. By Lemma \ref{temp_gap} we have both $\SG(n - r - 1) = 0$ and $\SG(n - r - 7) = 0$. Furthermore we must have $\SG(n - r - 2) = 1$, $\SG(n - r - 3) = 0$, $\SG(n - r - 4) = 1$, $\SG(n - r - 5) = 0$, $\SG(n - r - 8) = 1$, $\SG(n - r - 9) = 0$, $\SG(n - r - 10) = 1$, $\SG(n - r - 11) = 0$. Moreover $(n - r - 2)/2 \equiv 5 \pmod 6$ so by Table \ref{tab:mod_options} $\SG((n - r - 2)/2) = 0$. Similarly $(n - r - 10)/2 \equiv 1 \pmod 6$ so $\SG((n - r - 10)/2) = 0$.
Therefore $\SG((n - r - 4)/2) = 2$ and $\SG((n - r - 8)/2) = 2$, and therefore
\begin{equation}\label{nr84}
    \SG\left(\frac{n - r - 8}{4}\right) = 1.
\end{equation}

Now we claim that $\SG((n - r - 6)/2) = 1$. Indeed, otherwise we would have $\SG((n - r - 6)/2) = 0$,  $\SG((n - r - 4)/4) = 1$ contradicting (\ref{nr84}).

Therefore $\SG(n - r - 6) = 3$ and $\SG((n - r - 6)/3) = 2$ and $\SG((n - r - 3)/3) = 0$, since the only successor of $(n - r - 3) / 3$ is $(n - r - 6) / 3$. This contradicts the fact that $\SG(n - r - 3) = 0$.
\end{proof}

\section{Discussion}

A precise characterization of the SG function of $\imark(\{s\},\{d\})$ for $s,d$ relatively prime is still missing.

Regarding the game $\imark(\{1\},\{2,3\})$, the existence of the constant $c_3$ could probably be established in a very tedious way by a similar but more extensive case analysis. In any case, our partial results on $\imark(\{1\},\{2,3\})$ do not provide any concrete help in playing this game. A more pressing open problem is to find a polynomial-time algorithm for its SG function.

\paragraph{Acknowledgements.} One of the referees carried out the computation of Table~\ref{tab:max_gap} up to $n=943\,700\,000\,000$, using a Mac Pro with 1.5TB of RAM, and found that the maximum gaps did not change. We thank the referees for this and their other helpful comments.

\bibliographystyle{plainurl}
\bibliography{some}

\end{document}